\documentclass{amsart} 
\usepackage{amscd,amssymb,amsmath,amsbsy,amsthm,amssymb,mathtools}
\usepackage{color,subfig}
\usepackage[colorlinks,plainpages,backref,urlcolor=blue]{hyperref}
\usepackage[width=5.65in,height=8.0in,centering]{geometry}
\usepackage[mathscr]{euscript}

\usepackage{tikz}
\usetikzlibrary{calc,intersections,arrows,patterns,cd,shapes.misc}
\usepackage{booktabs}
\usepackage{enumerate}
\usepackage{enumitem}

\usepackage{cleveref} 

\makeatletter
\newcommand{\mylabel}[2]{(#2)\def\@currentlabel{#2}\label{#1}}
\makeatother

\theoremstyle{definition}
\newtheorem{theorem}{Theorem}[subsection]
\newtheorem{definition}[theorem]{Definition}
\newtheorem{lemma}[theorem]{Lemma}
\newtheorem{corollary}[theorem]{Corollary}

\newtheorem{proposition}[theorem]{Proposition}
\newtheorem{propdef}[theorem]{Proposition/Definition}

\newtheorem{exm}[theorem]{Example}
\newtheorem{rem}[theorem]{Remark}

\newtheorem{question}[theorem]{Question}

\newenvironment{example}%
{\pushQED{\qed}\begin{exm}}%
{\popQED\end{exm}}  

\newenvironment{remark}%
{\pushQED{\qed}\begin{rem}}%
{\popQED\end{rem}}  

\newcommand{\set}[1]{\left\{#1\right\}}   
\newcommand{\angl}[1]{\left<#1\right>}    

\newcommand{\A}{\mathscr{A}}
\newcommand{\X}{\mathscr{X}}   
\newcommand{\Z}{\mathbb{Z}}
\newcommand{\NN}{\mathbb{N}}
\renewcommand{\AA}{\mathbb{A}}   
\renewcommand{\P}{\mathbb{P}}

\newcommand{\C}{\mathbb{C}}
\newcommand{\R}{\mathbb{R}}
\newcommand{\T}{\mathbb{T}}
\renewcommand{\k}{\Bbbk}
\newcommand{\M}{\mathsf{M}}    
\newcommand{\typeA}{\mathsf{A}}  

\newcommand{\cyc}{\mathcal{Z}} 

\DeclareMathOperator{\der}{{Der}}
\DeclareMathOperator{\derbar}{\overline{Der}}
\DeclareMathOperator{\Tor}{{Tor}}
\DeclareMathOperator{\spec}{{Spec}}
\DeclareMathOperator{\pdim}{{pdim}}
\DeclareMathOperator{\codim}{{codim}}  
\DeclareMathOperator{\cl}{{cl}}  
\DeclareMathOperator{\rank}{{rk}}  
\DeclareMathOperator{\Sym}{Sym}
\DeclareMathOperator{\Ext}{Ext}

\newcommand{\uppercaseFactory}[2]{
    \foreach \x in {A,...,Z}{
        \expandafter\xdef\csname #1\x\endcsname{#2{\x}}}}
\newcommand{\lowercaseFactory}[2]{
    \foreach \x in {a,...,z}{
        \expandafter\xdef\csname #1\x\endcsname{#2{\x}}}}

\uppercaseFactory{sh}{\noexpand\mathcal}
\uppercaseFactory{cal}{\noexpand\mathcal}
\uppercaseFactory{scr}{\noexpand\mathscr}
\uppercaseFactory{mod}{}
\uppercaseFactory{fr}{\noexpand\mathfrak}
\lowercaseFactory{fr}{\noexpand\mathfrak}
\uppercaseFactory{rm}{\noexpand\mathrm}
\lowercaseFactory{rm}{\noexpand\mathrm}


\begin{document}

\begin{abstract}
    We study the Hadamard product of the linear forms defining a
    hyperplane arrangement with those of its dual, which we view as
    generating an ideal in a certain polynomial ring. We use this
    ideal, which we call the ideal of pairs, to study logarithmic
    derivations and critical set varieties of arrangements in a way
    which is symmetric with respect to matroid duality.  Our main
    result exhibits the variety of the ideal of pairs as a subspace
    arrangement whose components correspond to cyclic flats of the
    arrangement. As a corollary, we are able to give geometric
    explanations of some freeness and projective dimension results due
    to Ziegler and Kung--Schenck.
\end{abstract}

\title[Geometry of logarithmic derivations of hyperplane arrangements]%
      {Geometry of logarithmic derivations of hyperplane arrangements}

\author[G. Denham]{Graham Denham$^1$} 
\address{Department of Mathematics, University of Western Ontario,
London, ON  N6A 5B7, Canada}  
\email{\href{mailto:gdenham@uwo.ca}{gdenham@uwo.ca}}
\urladdr{\href{http://gdenham.math.uwo.ca/}%
{http://www.math.uwo.ca/\~{}gdenham}}
\thanks{$^1$Partially supported by an NSERC Discovery Grant (Canada)}

\author[A.\ Steiner]{Avi Steiner}
\address{Department of Mathematics, University of Western Ontario,
London, ON  N6A 5B7, Canada}  
\email{\href{mailto:asteine@uwo.ca}{asteine@uwo.ca}}
\urladdr{\href{https://sites.google.com/view/avi-steiner/}%
{https://sites.google.com/view/avi-steiner}}

\setcounter{tocdepth}{2}
\maketitle
\tableofcontents

\section{Introduction}
One of the themes in the study of hyperplane arrangements is the
relationship between their geometric properties and the combinatorics
of their underlying matroids.  In this article, we revisit the
logarithmic differentials on a hyperplane arrangement
and their relationship with
the critical set or maximum likelihood variety of the arrangement.

We consider these objects from a unifying point of view that brings to
light some new algebra and geometry of matroid realizations.  To be more
precise,
let $\M$ be a matroid on a ground set $E$, and $V=\k^E$ the $\k$-vector space
on $E$.  A linear realization of $\M$ over a field $\k$ is a linear
subspace $W\subseteq V$ for which the rank of a subset $S\subseteq E$ in $\M$
equals the rank of the projection of $W$ onto the coordinate subspace $\k^S$.
Let $W^\perp$ denote the $\k$-dual of $V/W$, a subspace of $V^*$.  Then
$W^\perp$ is a linear realization of the dual matroid, $\M^\perp$.
The graph of the evaluation pairing is, by definition, the map
$V^*\times V\to V\times V$ given by $(\varphi,v)\mapsto(v,\varphi(v))$.
Let $\X(W)$ denote the image of its restriction to $W\times W^\perp$.  

In slightly different language, it was noted in \cite[\S5.3]{D14}, \cite{HS14}
that $\X(W)$ is the affine critical set variety associated with $W$.
This is a construction that arose both in mathematical physics applications
of hyperplane arrangements \cite{Var95,OT95}, as well as in maximal
likelihood estimation in algebraic statistics \cite{CHKS06}.  We recall that
if $\M$ has no loops, a linear realization $W$ of $\M$ determines a hyperplane
arrangement $\A=\set{H_1,\ldots,H_n}$ in $W$, where $H_i$ is the intersection
of the $i$th coordinate hyperplane of $V$ with $W$.  Here, $\M$ is called
the matroid of the arrangement $\A$, and all hyperplane arrangements arise in
this way.  Maximum likelihood
estimation motivates the problem of determining the
zero locus of the logarithmic form
$\omega_\lambda:=d\log\big(\prod_{i=1}^n f_i^{\lambda_i}\big)$,
where $\lambda\in\Z^n$ is a lattice vector and each $f_i$ is a linear
function on a finite-dimensional vector space $W$.
For any $f_i$'s
and generic $\lambda$, Orlik and Terao~\cite{OT95} proved a conjecture of
Varchenko~\cite{Var95}, 
that the (projective) zero locus consists of $\beta(\M)$-many isolated
points, where $\beta(\M)$ is Crapo's $\beta$ invariant.

A connection between $\X(W)$ and logarithmic forms appeared first in
\cite{CDFV11}.  There, it is shown that $\X(W)$ is cut out set-theoretically by an ideal
$I_{\log}$ generated by applying
logarithmic differentials to a certain $1$-form.  The ideal is a complete
intersection if and only if $\A$ is a free arrangement: that is, the
module of logarithmic derivations $\der(\A)$ is free.  It is shown in
\cite{CDFV11} that the ideal $I_{\log}$ is arithmetically
Cohen--Macaulay under the weaker homological hypothesis that $\A$ is a tame
arrangement.  In general, $I_{\log}$ need not be radical.  The relationship
between the matroid $\M$, the variety $\X(W)$ and homological
properties of $\der(\A)$ is somewhat delicate, and a complete understanding
would also settle Terao's long-standing Freeness Conjecture.

Another motivation comes from recent work of Ardila, Denham and Huh~\cite{ADH20}
in which a key ingredient, the conormal fan of a matroid, is a tropical
analogue of the variety $\X(W)$.  A main result there \cite[Thm.\ 1.2]{ADH20}
is a substantial generalization of Orlik and Terao's proof of Varchenko's
conjecture, mentioned above.  With this in mind, we feel that further
investigation could also lead to additional combinatorial applications down
the road.

\subsection{Organization and main results}
In \S\ref{sec:prelim}, we review notation for arrangements, matroids and
critical sets.  For
each matroid realization $W$, we consider the ideal $\fra$ generated by
the pairwise products of linear forms $(f_1,\ldots,f_n)$ defining $W$ and
linear forms $(g_1,\ldots,g_n)$ defining $W^\perp$.  We call this the
\emph{ideal of pairs}.  We observe in \S\ref{sec:fra} that it is
closely related to the defining ideal $I_{\X}$ of the critical set variety
$\X(W)$, though it has the advantage of being symmetric under matroid
duality.  In \S\ref{subsec:log}, we recall the notion of logarithmic forms
on hyperplane arrangements, and we relate
graded submodules of the ideals $I_{\log}\subseteq I_{\X}$
with logarithmic derivations (Proposition~\ref{prop:a-deg1}), strengthening
a result from \cite{CDFV11}.

Our main results appear in Section \S\ref{sec:fra}.  In \S\ref{subsec:syz},
we show that logarithmic derivations for $\A$ and the dual arrangement
$\A^\perp$ appear as certain syzygies of the ideal $\fra$.  The affine
variety $V(\fra)$ is a subspace arrangement (by a result of Derksen and
Sidman~\cite[Prop.\ 4.5]{DS04}).  In \S\ref{ss:minprimes}, we give a
combinatorial description of those subspaces: it turns out that they are
products of linear spaces from $W$ and $W^\perp$, respectively,
corresponding to a cyclic flat from $\A$ and the complementary cyclic
flat from $\A^\perp$ (Theorem~\ref{thm:min_primes}).

In general, though, the ideal $\fra$ is not reduced, and while we understand
the minimal primes of $\fra$, the embedded primes remain somewhat mysterious.
We give some examples, and in \S\ref{subsec:uniform}, we characterize what happens for uniform matroids.  In \S\ref{subsec:noniso}, we use the geometry
of $V(\fra)$ to give an easy condition for an isomorphism type of the module
of logarithmic derivations to change for different realizations of the same
matroid.

As another application, in \S\ref{subsec:free} we recall necessary
combinatorial conditions for $\der(\A)$ to be free due to 
Kung and Schenck~\cite{KS06} and Ziegler~\cite{Zie89}, and we show that
they follow in a straightforward way from our description of $V(\fra)$.
Finally, we return to the question of when the inclusion $I_{\log}\subseteq
I_{\X}$ is an equality, and we show (Proposition~\ref{fra-lt}) that this is
the case when $\fra$ is an ideal of linear type.

\subsection{Acknowledgements}
The authors would like to thank Takuro Abe, Joseph Bonin and Uli Walther
for some helpful explanations.

\section{Preliminaries}\label{sec:prelim}
Let $\M$ be a matroid of rank $r$ on the ground set $[n]\coloneqq \{1,2,\ldots,n\}$, and let
$f\colon W\hookrightarrow \k^{[n]}$ be the inclusion of an $r$-dimensional
linear subspace over a field $\k$.
By definition, $W$ is a \emph{($\k$-)linear realization} of $\M$
provided the composite with each coordinate projection,
\[
W\to\k^{[n]}\twoheadrightarrow\k^B,
\]
is an isomorphism if and only if $B\subseteq[n]$ is a basis of $\M$.

Using the standard coordinate basis to identify $\k^{[n]}$ with its dual,
we regard
\begin{equation*}
    W^\perp \coloneqq (\k^{[n]}/W)^*
\end{equation*}
as a subspace of $\k^{[n]}$ as well,
and let $g$ denote this linear embedding.  
$W^\perp$ is a linear realization of the dual matroid $\M^\perp$.  We refer
to \cite{Oxbook} for terminology and background on matroid theory.

The remainder of this section proceeds as follows: In \S\ref{subsec:rings}, we fix some additional notation that will be used throughout this article. In \S\ref{subsec:arrs}, we recall the definition of the hyperplane arrangement $\A$ associated to the realization $W$. In \Cref{prop:coordsofX} of \S\ref{subsec:csv}, we show how to represent the critical set variety of an arrangement as the spectrum of a subalgebra of $R\otimes_\k R^\perp$. This representation will be used in \S\ref{sec:fra} as motivation to define the main tool of this paper: the ideal of pairs $\fra$. In \S\ref{subsec:log}, we relate the module of derivations $\der(\A)$ with certain ideals $I_{\log}$ and $I_{\X}$---this relationship will be reformulated in terms of the ideal of pairs at the start of \S\ref{subsec:syz}. We end this section with \S\ref{subsec:ass}, which contains some technical commutative algebra results which will be needed later on.

\subsection{Notation for some rings}\label{subsec:rings}
We make the following definitions:
\begin{itemize}
    \item $x_1,\ldots,x_n$ -- the coordinate functions on $\k^{[n]}$
    \item $y_1,\ldots,y_n$ -- the dual coordinate functions on $(\k^{[n]})^*$
    \item $S\coloneqq \k[x_1,\ldots,x_n]$ -- the coordinate ring of $\k^{[n]}$
    \item $\T^n$ -- the torus in $\k^{[n]}$ given by the equation $x_1\cdots x_n\neq 0$
    \item $f_i$ -- the restriction of $x_i$ to $W$ for each $1\leq i\leq n$
    \item $g_i$ -- the restriction of $y_i$ to $W^\perp$ for each $1\leq i\leq n$
    \item $R\coloneqq \k[W]=\k[f_1,\ldots,f_n]$ and $R^\perp\coloneqq \k[W^\perp]=\k[g_1,\ldots,g_n]$
\end{itemize}

Since $\k^{[n]}\cong W\oplus W^\perp$, we have $S=R\otimes_{\k}R^\perp$.
We will also make use of a ring of parameters
\[A:=\k[a_1,\ldots,a_n],\]
writing $R[a]$ and $S[a]$ interchangably for $R\otimes_\k A$ and $S\otimes_\k A$, respectively.  The standard gradings on $R$, $R^\perp$, and $A$ give multigradings: in order to distinguish the bigrading on $S$ from that of $R[a]$, we will write the $A$-degree last and separate it with a semicolon.


\subsection{Hyperplane arrangements}\label{subsec:arrs}

If $\M$ has no loops, then each $f_i\neq 0$,
and we obtain a hyperplane arrangement $\A$ from $W$ with hyperplanes
\[
\set{H_i:=\ker f_i : 1\leq i\leq n}.
\]
Similarly, if $\M$ has no coloops, then each $g_i\neq 0$, and we obtain a
dual hyperplane arrangement $\A^\perp$ from $W^\perp$ in the same way.
Given a hyperplane arrangement $\A$, we let
\begin{align*}
  U(\A)& := W-\bigcup_{i\in[n]}H_i\\
  & = W\cap \T^n
\end{align*}
denote its complement, and let $\P U(\A)\subseteq \P W$ denote its
quotient by the diagonal action of $\k^{\times}$.

\subsection{The critical set variety}\label{subsec:csv}
From various points of view it is of interest to consider rational
functions with poles and zeroes on hyperplanes.  Here, we take $\k=\C$.

Given a lattice vector $\lambda\in\Z^{[n]}$ and a realization
$f\colon W\to\C^{[n]}$ of a matroid without loops, let
\[
\Phi_\lambda\coloneqq\prod_{i\in[n]} f_i^{\lambda_i},
\]
regarded as a function $\Phi_\lambda\colon U(\A)\to\C^\times$.  If
$\sum_i \lambda_i=0$, then $\Phi_\lambda$ induces a well-defined function
on $\P U(\A)$ as well.

By definition, the (affine) critical set variety $\X(W)\subseteq
W\times \C^{[n]}\subseteq \C^{[n]}\times\C^{[n]}$ is the Zariski closure
of the pairs $(p,\lambda)$ for which $p$ is a critical point of the function
$\Phi_{\lambda}$.  We refer to \cite{CDFV11,DGS11} and \cite{Huh13} for
more details.  The critical set variety is also an instance of a maximal
likelihood variety---see, for example, \cite{HS14}.

If one restricts $p$ to $U(\A)$, then $p$ is a critical point of
$\Phi_\lambda$ if and only if the $1$-form
\begin{align*}
  \omega_\lambda  &\coloneqq d\log\Phi_\lambda = \sum_{i=1}^n\lambda_i df_i/f_i
\end{align*}
vanishes at $p$.  A calculation shows that this is the case precisely when $(\lambda_1/f_1(p), \ldots, \lambda_n/f_n(p))$ is in $W^\perp$, which is to say that there exists a $q$ such that for all $i$, $\lambda_i=f_i(p)g_i(q)$.  Then $\X(W)\cap(\T^n\times\T^n)$ is parameterized by
\begin{equation}\label{eq:f,fg}
(f_1(p),f_2(p),\ldots,f_n(p),f_1(p)g_1(q),\ldots,f_n(p)g_n(q))
\end{equation}
for $(p,q)\in W\times W^\perp$.
\begin{proposition}\label{prop:coordsofX}
  For any realization $W$ of a matroid without loops, we have 
  \begin{align*}
    \X(W)&=\spec\C[f_1,\ldots,f_n,f_1g_1,\ldots,f_ng_n]\\
    &=\spec R[f_1g_1,\ldots,f_ng_n].
  \end{align*}
\end{proposition}
\begin{proof}
  Let $\X'$ denote
  the spectrum of the domain $R[f_1g_1,\ldots,f_ng_n]$, viewed as a closed subscheme of $\C^{[n]}\times\C^{[n]}$ via the $\C$-algebra map $\C[s_1,\ldots,s_n,t_1,\ldots,t_n]\to R[f_1g_1,\ldots,f_ng_n]$ induced by $s_i\mapsto f_i$ and $t_i\mapsto f_ig_i$ $(i=1,\ldots,n)$.  The intersections
  of $\X'$ and $\X(W)$ with $\T^n\times\T^n$ agree.
  By construction, the torus $\T^n\times\T^n$ is dense in $\X(W)$: taking
  closures, we see $\X'$ contains $\X(W)$.  Clearly $\X'$ is irreducible.
  By inverting each $f_i$ we see $\dim\X'=r+(n-r)=n$.  On the other hand,
  $\X(W)$ is also irreducible by \cite[Cor.\ 2.10]{CDFV11}) and has the
  same dimension, so they are equal.
\end{proof}
If $W$ realizes a matroid with loops, we can take the equality above as
the definition of $\X(W)$.
Clearly, deleting the loops leaves the algebra unchanged.  We can also
define $\X(W)$ over an arbitrary field $\k$ in this way.

The defining ideal of $\X(W)$ will be important in what follows.
\begin{definition}\label{def:IX}
The subalgebra $R[f_1g_1,\ldots,f_ng_n]$ above is the image of a ring homomorphism $R[a]\to
S$ given by sending $a_i\mapsto f_ig_i$ for $1\leq i\leq n$.  We let $I_{\X}$
denote the kernel, a prime ideal of codimension $n$ in $R[a]$.  We note
that $I_{\X}$ inherits the standard bigrading of $R[a]$.
\end{definition}
%

\subsection{Critical sets and logarithmic forms}\label{subsec:log}

Let $\der(\A)$ denote the $R$-module of logarithmic derivations on $\A$, and
$\Omega^1(\A)=\der(\A)^\vee$ the dual module of logarithmic $1$-forms --- see,
e.g., \cite{OTbook} for details.  The logarithmic forms appear naturally
in relation to the critical set variety.  We will grade derivations so that
$\deg(\partial/\partial x_i)=-1$ for each $i$.
\begin{theorem}[\cite{CDFV11}]
  The variety $\X(W)$ is the zero set of the ideal of $R[a]$ given
  by 
\[
  I_{\log}\coloneqq \left(\angl{\theta,\omega_a}\colon \theta\in\der(\A)\right),
  \]
  where $\omega_a\coloneqq\sum_i a_idf_i/f_i$.%
\footnote{
Equivalently, $I_{\log}$ is generated by applying logarithmic derivations to
the formal expression $\log\big(\prod_{i=1}^nf_i^{a_i}\big)$.}
\end{theorem}

Recall that an arrangement $\A$ is said to be \emph{free} if $\der(\A)$ is a
free $R$-module.  In \cite{CDFV11}, it is shown that $\A$ is free if and only
if $I_{\log}$ is a complete intersection.
A weaker but quite useful property goes back to
\cite{OT95b}: an arrangement $\A$ is \emph{tame} if $\pdim\Omega^p(\A)\leq p$ for
$1\leq p\leq r$, where $\Omega^p(\A)\cong\big(\bigwedge^p\Omega^1(\A)\big)^{\vee\vee}$
is the module of logarithmic $p$-forms (see \cite{DS12}.)

The Theorem shows $I_{\log}\subseteq I_{\X}$, but equality fails in general --- see Example~\ref{ex:bracelet} below.
However, if $\A$ is tame, then $I_{\log}=I_{\X}$ (\cite{CDFV11}).

\begin{proposition}\label{prop:a-deg1}
  As graded $R$-modules, 
  \[
  (I_{\log})_{(\cdot;1)}=(I_{\X})_{(\cdot;1)}\cong\der(\A)(-1).
  \]
\end{proposition}
\begin{proof}
  Noting that $I_{\log}\subseteq I_{\X}$, we will construct maps
  $(I_{\X})_{(p;1)}\to\der(\A)_{p-1}\to  (I_{\log})_{(p;1)}$ which compose
  to the inclusion $I_{\X}\subseteq I_{\log}$, for each integer $p$.
  Any $u\in (I_{\X})_{(p;1)}$ can be written
  \begin{equation}\label{eq:syz}
  u=\sum_{i=1}^n c_ia_i \quad\text{where}\quad
  \sum_{i=1}^n c_if_ig_i=0,
  \end{equation}
  for some homogeneous elements $c_i\in R$.
  By reordering the ground set of $\M$ if necessary, we assume that $[r]$ is
  independent in $\M$, so that we may choose bases for $W$ and $W^\perp$,
  respectively, with $f_i=x_i$ for $1\leq i\leq r$ and $g_i=x_i$ for
  $r+1\leq i\leq n$.

  Let $\theta=\sum_{i=1}^rc_ix_i\partial/\partial x_i$.  We claim $\theta(f_j)=
  c_jf_j$ for all $j$, so $\theta\in\der(\A)$.  If so, 
\[
\angl{\theta,\omega_a}=\sum_{i=1}^na_i\theta(f_i)/f_i =u,
\]
which implies $u\in I_{\log}$.

For $1\leq j\leq r$ the claim is obvious.  For
$r+1\leq j\leq n$, we compute as follows.
  By orthogonality, $\partial f_j/\partial x_i =
  -\partial g_i/\partial x_j$ for all $1\leq i\leq r$ and $r+1\leq j\leq n$.
  Differentiating \eqref{eq:syz} by $x_j$ shows, for each $j$,
  \begin{align*}
    c_jf_j &= -\sum_{i=1}^r c_if_i\frac{\partial g_i}{\partial x_j} \\
    &= \sum_{i=1}^r c_i f_i \frac{\partial f_j}{\partial x_i}\\
    &= \theta(f_j),
  \end{align*}
  as required.
\end{proof}

\begin{example}\label{ex:bracelet}
  Consider the arrangement $\A$ of $9$ hyperplanes in $\C^4$ defined by the
  columns of the matrix
\[
\setcounter{MaxMatrixCols}{20}
\begin{pmatrix}
1 & 0 & 0 & 1 & 0 & 0 & 1 & 1 & 0 \\
0 & 1 & 0 & 0 & 1 & 0 & 1 & 0 & 1 \\
0 & 0 & 1 & 0 & 0 & 1 & 0 & 1 & 1 \\
0 & 0 & 0 & 1 & 1 & 1 & 1 & 1 & 1 \\
\end{pmatrix}.
\]
This arrangement has $\pdim\Omega^1(\A)=2$, so it is not tame.  A calculation
with Macaulay2~\cite{M2} shows that $I_{\log}$ has an embedded prime at the
origin.

Clearly $I_{\log}$ is always generated by elements of bidegree $(p;1)$, for
integers $p\geq 0$.  For this arrangement, however, computation also shows
that $I_{\X}$ has a generator of degree $(2;2)$.
\end{example}


\subsection{Associated primes and slicing}\label{subsec:ass}
We end this section with some technical results which we will use to prove \Cref{cor:min_primes2}.  

\begin{lemma}\label{prim-restr}
    Let $S$ be a bigraded Noetherian ring, $R\coloneqq S_{(\cdot,0)}$. Let $k\in \Z$, and let $M$ be a graded $S$-module. Let $\frp$ be a homogeneous prime ideal in $S$. If $N$ is a $\frp$-primary graded submodule of $M$ and $N_{(\cdot,k)}\neq M_{(\cdot,k)}$, then $N_{(\cdot,k)}$ is a $\frp_{(\cdot,0)}$-primary submodule of $M_{(\cdot,k)}$.
\end{lemma}
\begin{proof}
    Let $x\in M_{(\cdot,k)}$ with $x\notin N_{(\cdot,k)}$, and let $a\in R$ be homogeneous. Suppose $ax\in N_{(\cdot,k)}$. Then $ax\in N$, so since $N$ is $\frp$-primary and $x\notin N$, $a\in \frp$. Thus, $a \in \frp_{(\cdot,0)}$. Hence, $N_{(\cdot,k)}$ is $\frp_{(\cdot,0)}$-primary.
\end{proof}

\begin{example}
    The condition $N_{(\cdot,k)}\neq M_{(\cdot,k)}$ in \Cref{prim-restr} is necessary: Let $S=\C[x,y]$ be standard bigraded, $M=S$, and $N=\frp=(x,y)$. Then $M_{(\cdot,1)}=N_{(\cdot,1)}=R$, so in particular $M_{(\cdot,1)}/N_{(\cdot,1)}$ has no associated primes.
\end{example}

\begin{corollary}\label{ass-restr}
    Let $S$ be a bigraded Noetherian ring, $R\coloneqq S_{(\cdot,0)}$. Let $k\in \Z$, let $M$ be a graded $S$-module, let $\frp_1,\ldots,\frp_s$ be the associated primes of $M$, and let $N_1\cap\cdots \cap N_s$ be a minimal (homogeneous) primary decomposition of $0$ in $M$, where $N_i$ is $\frp_i$-primary.
    \begin{enumerate}[label=\textnormal{(\alph*)}]
        \item\label{ass-restr.ass} The associated primes of $M_{(\cdot,k)}$ are contained in $\{(\frp_i)_{(\cdot,0)} : (N_i)_{(\cdot,k)}\neq M_{(\cdot,k)}\}$, hence also in $\{\frp_{(\cdot,0)} : \frp\text{ an associated prime of }M\}$.
        
        \item\label{ass-restr.min} The minimal primes of $M_{(\cdot,k)}$ are $\min\{(\frp_i)_{(\cdot,0)} : (N_i)_{(\cdot,k)}\neq M_{(\cdot,k)}\}$.
    \end{enumerate}
\end{corollary}
\begin{proof}
    By \Cref{prim-restr}, $\bigcap_i (N_i)_{(\cdot,k)}$, where $i$ runs through those indices for which $(N_i)_{(\cdot,k)}\neq M_{(\cdot,k)}$, is a (not necessarily minimal) primary decomposition of $0$ in $M_{(\cdot,k)}$. Now use standard facts relating the associated primes of $M$ to primary decompositions of $0$ in $M$ to get \ref{ass-restr.ass}. 
    
    We now prove \ref{ass-restr.min}. A standard fact about minimal primes says that if $L$ is a finitely-generated $R$-module, $\frq_1,\ldots,\frq_t$ are (not necessarily distinct) primes in $R$, and $K_1\cap\cdots \cap K_t$ is a not necessarily minimal primary decomposition of $0$ in $L$, then the minimal primes of $L$ are exactly the minimal elements of $\{\frq_1,\ldots,\frq_t\}$. Now use that $\bigcap_i (N_i)_{(\cdot,k)}$, where $i$ runs through those indices for which $(N_i)_{(\cdot,k)}\neq M_{(\cdot,k)}$, is a primary decomposition of $0$ in $M_{(\cdot,k)}$, and that the relevant $(N_i)_{(\cdot,k)}$ are $(\frp_i)_{(\cdot,0)}$-primary.
\end{proof}

\section{The ideal of pairs}\label{sec:fra}
In this section, we introduce and study the ideal of pairs associated with
a pair of matroid realizations $(W,W^\perp)$.  This ideal determines a
subspace arrangement which we are able to understand set-theoretically but
not scheme-theoretically.  Syzygies of the ideal of pairs contain information
about logarithmic derivations on both hyperplane arrangements $\A$ and
$\A^\perp$.  

\begin{definition}
    The \emph{ideal of pairs} associated to a realization $W$ is
the ideal
\[
\fra\coloneqq \fra_{W,W^\perp}\coloneqq(f_1g_1,\ldots,f_ng_n)
\]
of $S$.
\end{definition}

%
Clearly,
 \[
 \fra_{(\cdot,1)} \cong  \bigoplus_{i\geq0}\left(R[a]/I_{\X}\right)_{(i;1)} = (R[a]/I_\X)_{(\cdot;1)}.
 \]

\subsection{Syzygies of \texorpdfstring{$\fra$}{the ideal of pairs}}\label{subsec:syz}
Relations amongst the products $\set{f_ig_i}$ are related to the discussion
in Section~\S\ref{subsec:log}, and Proposition~\ref{prop:a-deg1} can be reformulated as follows.  Let $K$ denote the kernel of the homomorphism
$S[a]_{(\cdot,\cdot;1)}(-1,-1)\to \fra$ sending $a_i$ to $f_ig_i$.
\begin{theorem}\label{thm:slices}
  Let $W$ be a subspace realizing a matroid $\M$.
  Let $\Pi$ be the partition of $[n]$ given by the connected components of $\M$.
  \begin{itemize}
  \item $K_{(1,1)}$ is spanned by sums $\sum_{i\in C}a_i$, for each
    connected component $C$ in $\Pi$.
  \item $K_{(\cdot,1)}\cong \der(\A)(-1)$ as graded $R$-modules, and
  \item $K_{(1,\cdot)}\cong \der(\A^\perp)(-1)$ as graded $R^{\perp}$-modules.
  \end{itemize}
\end{theorem}
\begin{proof}
  As $R$-modules, we have $(I_{\X})_{(\cdot;1)}\cong K_{(\cdot,1)}$ by
  construction. So the second claim follows by Proposition~\ref{prop:a-deg1},
  and the third claim by exchanging the roles of $R$ and $R^\perp$.

  To establish the first claim, we note that
  the only degree-zero derivations on a connected
  component $C$ are multiples of the Euler derivation, $\sum_{i\in C}x_i
  \partial/\partial x_i$.  Now $\der(\A)$ is a direct sum of logarithmic
  derivations on the connected components (see, e.g.,~\cite{OTbook}), and
  the isomorphism $K_{(1,1)}\cong\der(\A)_0$ sends $a_i$ to
  $x_i\partial/\partial x_i$.  
\end{proof}

\subsubsection{Minimal syzygies}
The claim about $K_{(\cdot,1)}$ in \Cref{thm:slices} can be reformulated as saying that we have exact sequences
\begin{equation}\label{eq:3-terms-old}
\begin{tikzcd}[column sep=small]
  0\ar[r] & \der(\A)(-1)\ar[r] & F\ar[r] & \fra_{(\cdot,1)} \ar[r]&  0,
\end{tikzcd}
\end{equation}
and
\begin{equation}\label{eq:4-terms-old}
\begin{tikzcd}[column sep=small]
  0\ar[r] & \der(\A)(-1)\ar[r] & F\ar[r] & E \ar[r] & (S/\fra)_{(\cdot,1)}\ar[r] & 0
\end{tikzcd}
\end{equation}
where $E=S_{(\cdot,1)}\cong R\otimes_{\k} (W^\perp)^*$ is a free $R$-module of rank $n-r$, and $F=S[a]_{(\cdot,0;1)}(-1)=R[a]_{(\cdot,1)}(-1)$ is a free $R$-module of rank $n$.

We can refine these exact sequences slightly using \Cref{thm:slices}. 
Let $\theta_1,\ldots,\theta_\kappa$ be the Euler derivations from the proof of \Cref{thm:slices}. Define\footnote{In characteristic 0, this module is canonically isomorphic to the submodule\[\der_0(\A)\coloneqq \{\theta\in \der(\calA) : \theta(f_i)=0\text{ for all }i\}\] of $\der(\A)$.}
\begin{equation*}
    \derbar(\A) \coloneqq \frac{\der(\A)}{R\{\theta_1,\ldots,\theta_\kappa\}}.
\end{equation*}
When the characteristic of $\k$ is 0, the exact sequence
\begin{equation}\label{eq:der-der0}
    0 \to  R^\kappa \xrightarrow{[\theta_1\;\cdots \;\theta_k]} \der(\calA) \to \derbar(\calA) \to 0
\end{equation}
is canonically split (see, e.g., \cite[Rmk.~2.10(4)]{Wal17}).

Recall that a graded $R$-module $N$ is a \emph{minimal graded $k$th syzygy module} of the graded $R$-module $M$ if there exists an exact sequence
\[ \xi\colon N \xrightarrow{\partial_{k+1}} P_k \xrightarrow{\partial_k} \cdots \xrightarrow{\partial_1} P_0\]
such that
\begin{itemize}
    \item each $P_i$ is graded free;
    \item $\operatorname{coker}(\partial_1) \cong M$; and
    \item the differentials of $\k\otimes_R\xi$ are all zero, or equivalently, $\operatorname{im}(\partial_{i+1})\subseteq R_+P_i$ for each $0\leq i\leq k$, where $R_+$ is the homogeneous maximal ideal.
\end{itemize}

\begin{lemma}\label{ann-seqs}
    \hphantom{NEWLINE}
    \begin{enumerate}[label=\textnormal{(\alph*)}]
        \item\label{ann-seqs.free} The graded $R$-module $F/RK_{(1,1)}$ is graded free, where $F$ is the graded free $R$-module defined above, and $K$ is the module defined at the start of \S\ref{subsec:syz}.
        \item\label{ann-seqs.seqs} There are exact sequences
        \begin{equation}\label{eq:3-terms}
        \begin{tikzcd}[column sep=small]
          0\ar[r] & \derbar(\A)(-1)\ar[r] & F/RK_{(1,1)}\ar[r] & \fra_{(\cdot,1)} \ar[r]&  0,
        \end{tikzcd}
        \end{equation}
        and
        \begin{equation}\label{eq:4-terms}
        \begin{tikzcd}[column sep=small]
          0\ar[r] & \derbar(\A)(-1)\ar[r] & F/RK_{(1,1)}\ar[r] & E \ar[r] & (S/\fra)_{(\cdot,1)}\ar[r] & 0
        \end{tikzcd}
        \end{equation}
        where $E$ and $F$ are the graded free $R$-modules defined above. 
        
        \item\label{ann-seqs.syz} $\derbar(\A)(-1)$ is a minimal graded first and second syzygy module of $\fra_{(\cdot,1)}$ and $(S/\fra)_{(\cdot,1)}$, respectively.
    \end{enumerate}
\end{lemma}
\begin{proof}
    \ref{ann-seqs.free} By \Cref{thm:slices}, the vector space $K_{(1,1)}$ is spanned by the images of the Euler derivations $\theta_1,\ldots,\theta_\kappa$ corresponding to the connected components of $\M$. These images all live the vector space $F_1$ spanned by $a_1,\ldots,a_n$. Hence, $F/RK_{(1,1)} \cong R(-1)\otimes_\k (F_1/K_{(1,1)})$ is a graded free $R$-module.
    
    \ref{ann-seqs.seqs} By \Cref{thm:slices}, the vector space $K_{(1,1)}$ is in the kernel of the map $F\to \fra_{(\cdot,1)}$, and it is spanned by the images of the Euler derivations $\theta_1,\ldots,\theta_\kappa$ corresponding to the connected components of $\M$. Now use 
    the definition of $\derbar(\A)$.
    
    \ref{ann-seqs.syz} 
    Every degree zero element of $\der(\A)$ lives in the $R$-span of the Euler derivations. Therefore, the image of $\derbar(\A)(-1)$ in $F/RK_{(1,1)}$ lives in degrees at least 1, i.e.\ in $R_+\cdot(F/RK_{(1,1)})$, where $R_+$ is the homogeneous maximal ideal. 
    This shows that $\derbar(\A)$ is a minimal graded first syzygy module of $\fra_{(\cdot,1)}$. The claim about $(S/\fra)_{(\cdot,1)}$ follows once we observe that, by definition, the image  $\fra_{(\cdot,1)}$ of $F/RK_{(1,1)}\to E$ is contained in $R_+ E$.
\end{proof}

\begin{remark}\label{min-free}
    If $A$ is free, 
    then the proof of \Cref{ann-seqs}\ref{ann-seqs.syz} shows that \eqref{eq:3-terms} and \eqref{eq:4-terms} are minimal graded free resolutions of $\fra_{(\cdot,1)}$ and $(S/\fra)_{(\cdot,1)}$, respectively.
\end{remark}

\subsubsection{Tor of the module of derivations}
Recall that $W$ and $W^\perp$ are $r$- and $(n-r)$-dimensional, respectively.

\begin{corollary}\label{cor:tor_of_der}
  For all $p\geq1$ and $i\geq1$, we have
  \begin{align*}
    \Tor_{p+1}^S(\fra,\k)_{(i,1)}&\cong\Tor_p^R(\der(\A),\k)_{i-1}\quad
    \text{and}\\
  \Tor_{p+1}^S(\fra,\k)_{(1,i)}&\cong\Tor_p^{R^\perp}(\der(\A^\perp),\k)_{i-1}.
  \end{align*}
  Additionally, for all $i\geq 1$, we have
  \begin{align*}
      \dim_\k\Tor_1^S(\fra,\k)_{(i,1)} &=  \dim_\k (\der(\A)\otimes_R\k)_{i-1} -  \binom{i+r-2}{r-1}\kappa,\\
      \dim_\k\Tor_1^S(\fra,\k)_{(1,i)} &=  \dim_\k (\der(\A^\perp)\otimes_R\k)_{i-1} -  \binom{i+n-r-2}{n-r-1}\kappa,\\
  \end{align*}
  where $\kappa$ is the number of components of $\M$.
\end{corollary}
\begin{proof}
    If $F$ is a graded free $S$-module, then $F_{(\cdot,1)}$ is a graded free $R$-module. Moreover, $(-)_{(\cdot,1)}$ is an exact functor. Thus, $(-)_{(\cdot,1)}$ takes graded free resolutions to graded free resolutions. Thus, the statement for $p\geq 1$ follows from \eqref{eq:3-terms-old}. The statement for $p=0$ follows from \eqref{eq:3-terms}, the definition of $\derbar(\A)$, and the formula for the Hilbert function of the polynomial rings $R$ and $R^\perp$.
\end{proof}
Thus the ideal of pairs ``sees'' in particular whether or not $\A$ or its
dual are free arrangements.
\begin{example}\label{ex:A3syz}
  Let $\A$ be the braid arrangement $\typeA_3$.  This is also the 
  rank-$3$ graphic arrangment defined by the complete graph
  $K_4$.  Here $n=6$, the arrangement $\A$ is free,
  and $\A^\perp$ is isomorphic to $\A$.
  The ideal $\fra$ has projective dimension $3$, and
  its bigraded Betti numbers are given by:
  \begin{equation}\label{eq:A3syz}
  \begin{array}{|lll}
    t & 5t^2 & t^3\\
    t & 7t & 5t^2 \\
    5 & t & t\\ \hline
  \end{array},
  \end{equation}
  where the $(i,j)$ entry is the polynomial
  \[
  \sum_{p\geq0}\dim_{\k}\big(
  \Tor^S_p(\fra,\k)_{(i,j)}\big)t^p.
  \]
  The root system $\typeA_3$ has
  coexponents $\set{1,2,3}$, so $\der(\A)\cong R\oplus R(-1)\oplus R(-2)$,
  and the two positive-degree generators of $\der(\A)$ are reflected in the
  bottom row of \eqref{eq:A3syz}.
\end{example}

\subsection{Minimal primes over \texorpdfstring{$\fra$}{the ideal of pairs}}\label{ss:minprimes}
The ideal $\fra$ defines a variety $V(\fra)$ in $\AA^n$ whose components
are all linear subspaces, by a result of Derksen and
Sidman~\cite[Prop.\ 4.5]{DS04}.  We describe these linear subspaces in \Cref{thm:min_primes}. We use this description in \Cref{cor:min_primes} to produce a combinatorial lower bound for $\pdim_S(S/\fra)$, and then again in \Cref{cor:min_primes2} to describe the minimal associated primes of $\fra_{(\cdot,1)}$.

\bigskip
Let $L(\M)$ denote the lattice of flats of $\M$ (ordered by inclusion).  It
is well known that $L(\M)$ is isomorphic to the lattice of subspaces cut out by
the hyperplane arrangement $\A$ (ordered by reverse inclusion).  

A flat
$F\in L(\M)$ is \emph{cyclic} if $F$ is a union of circuits or, equivalently,
if its complement $F^\complement\in L(\M^\perp)$.
Let $\cyc(\M)\subseteq L(\M)$ denote the subposet of cyclic flats.  Here is a
basic fact about cyclic flats.  (See, for example, \cite[Ex.\ 8.2.13]{Oxbook}.)
\begin{propdef}\label{propdef:cyclic_flats}
  For a flat $F\in L(\M)$, let
  \[Z(F)\coloneqq\bigcup_{C\subseteq F}C,\]
  where $C$  ranges over all circuits of $\M$.  Then $Z(F)$ is a cyclic flat, and
  $\cl_{\M^\perp}(F^\complement)=Z(F)^\complement$.
\end{propdef}

\begin{definition}
    We say a pair $(F,G)\in L(\M)\times L(\M^\perp)$ is
    a \emph{biflat} if 
    $F\cup G=[n]$.\footnote{This disagrees slightly with the language of
    \cite{ADH20}, where in addition $F$ and $G$ must be nonempty and not both equal to $[n]$.}
    Let $L(\M,\M^\perp)$
    denote the subposet of $L(\M)\times L(\M^\perp)$ consisting of just the biflats.
\end{definition}

For any subsets $I,J\subseteq[n]$, we let
\[
\frp_{I,J}\coloneqq (f_i\colon i\in I)+(g_j\colon j\in J),
\]
a linear ideal of $S$, and let $L_{I,J}\coloneqq V(\frp_{I,J})$.
\begin{remark}\label{rmk:spans}
  We have $\frp_{I,J}=\frp_{\cl_\M(I),\cl_{\M^\perp}(J)}$, since a linear
  functional $f_j$ is in the span of $\set{f_i\colon i\in I}$ if and only if
  $j$ is in the (matroid) span of $I$.  Clearly
  \[
  \codim \frp_{I,J} =
  \rank_{\M}(I)+\rank_{\M^\perp}(J).\qedhere
  \]
\end{remark}

\begin{theorem}\label{thm:min_primes}
  Let $W$ be a realization of a matroid $\M$ and $\fra$ its ideal of pairs.
  \begin{enumerate}[label=\textnormal{(\alph*)}]
    \item\label{thm:min_1}
  Every associated prime of $S/\fra$ is of the form $\frp_{F,G}$ for some biflat $(F,G)$.
\item\label{thm:min_2}
  The minimal primes of $S/\fra$ are $\set{\frp_{F,F^\complement}
    \colon F\in \cyc(\M)}$. In particular,
    \[ V(\fra) = \bigcup_{F\in \cyc(\M)} L_{F,F^\complement}.\]
\item\label{thm:min_3}
  If $F\in\cyc(\M)$, the primary component of $\fra$ corresponding to
  $\frp_{F,F^\complement}$ is $\frp_{F,F^\complement}$ itself.
  \end{enumerate}
  
\end{theorem}
\begin{proof}
  \ref{thm:min_1} Let $P$ be an associated prime of $\fra$.  By
  \cite[Prop.\ 4.5]{DS04}, $P$ is necessarily a linear ideal of the form
  $\frp_{F,G}$ for some subsets $F,G\subseteq [n]$.  By the previous \namecref{rmk:spans},
  we may assume $F\in L(\M)$ and $G\in L(\M^\perp)$.  That $F\cup G=[n]$
  follows from the fact that $P$ is prime and contains $f_ig_i$ for each $i$.

  \ref{thm:min_2} Let $P$ be a minimal prime of $S/\fra$.  By
  \ref{thm:min_1}, we have $P=\frp_{F,G}$ for some biflat $(F,G)$.
  By Proposition~\ref{propdef:cyclic_flats}, we have $Z(F)\subseteq F$ and
  $\cl_{\M^\perp}(F^\complement)=Z(F)^\complement$.

  Since $F^\complement \subseteq G$, taking closures in $\M^\perp$ shows
  $Z(F)^\complement\subseteq G$ as well.  By minimality, then, $F=Z(F)$ and
  $G=F^{\complement}$.

  \ref{thm:min_3} Localizing $\fra$ at $\frp=\frp_{F,F^\complement}$ gives
  the ideal $\frp S_{\frp}$.  Now we note that $\frp S_{\frp}\cap S=\frp$.  
\end{proof}

\begin{remark}\label{rmk:always-min}
    If $\M$ has no loops or coloops, then $[n]$ is a cyclic flat. Therefore, the prime ideals $\frp_{[n],\emptyset}=\langle W^*\rangle$ and $\frp_{\emptyset,[n]}=\langle(W^\perp)^*\rangle$ are minimal primes of $S/\fra$ by \Cref{thm:min_primes}\ref{thm:min_2}.
\end{remark}

\begin{example}\label{ex:unif}
  If $\M$ is a uniform matroid, then the associated primes of $S/\fra$ are $\frp_{[n],\emptyset}$, $\frp_{\emptyset,[n]}$, and $\frp_{[n],[n]}$. This will be proven in \Cref{cor:unif}. 
\end{example}

\begin{example}[\Cref{ex:bracelet} continued]\label{ex:bracelet-fra}
  A computation using Macaulay2 \cite{M2} shows that $S/\fra$ has no embedded primes.
\end{example}

\begin{example}\label{ex:7-hyps}
    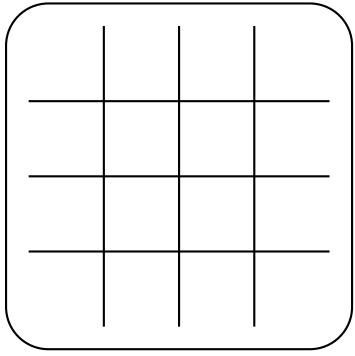
\begin{figure}
        \centering
        \begin{tikzpicture}
            \draw[thick,rounded corners=16pt] (-1.3,-1.3) rectangle (3.3,3.3);
            \draw[thick] (1,-1) -- (1,3);
            \draw[thick] (-1,1) -- (3,1);
            \draw[thick] (0,-1) -- (0,3); 
            \draw[thick] (-1,0) -- (3,0); 
            \draw[thick] (2,-1) -- (2,3);
            \draw[thick] (-1,2) -- (3,2);
        \end{tikzpicture}
        \caption{The projectivization of the arrangement from \Cref{ex:7-hyps} with the hyperplane $f_1=0$ placed at infinity.}
        \label{fig:7-hyps}
    \end{figure}
  Consider the arrangement $\A$ of 7 hyperplanes in $\C^3$ defined by the columns of the matrix
  \[ \begin{pmatrix}
    1&1&1&1&1&1&1\\
    0&1&0&2&0&3&0\\
    0&0&1&0&2&0&3
  \end{pmatrix}.\]
  This arrangement is pictured in \Cref{fig:7-hyps}. A computation using Macaulay2 \cite{M2} shows that the embedded primes correspond to the biflats $(1246,[7])$, $(1357,[7])$, and $([7],[7])$.
\end{example}

\begin{corollary}\label{cor:min_primes}
  For any arrangement $\A$, we have
\[
  \pdim_S(S/\fra) \geq \max\set{2\rank(F)-|F|+n-r\colon F\in \cyc(\M)}
\]
\end{corollary}
\begin{proof}
  Localization does not increase projective dimension, and the projective
  dimension of the residue field at $\frp_{F,F^\complement}$ equals the
  codimension, which by \cite[Prop.~2.1.9]{Oxbook} is
  \[
  \rank_{\M}(F)+\rank_{\M^\perp}(F^\complement) = 2\rank_{\M}(F)-|F|+n-r.\qedhere
  \]
\end{proof}
We note that $\pdim_S(S/\fra)\geq n-r$: since we assume the underlying matroid
has no loops, $F=\emptyset$ is a cyclic flat.  If, additionally, $\M$ has no
coloops, then $\pdim_S(S/\fra)\geq r$, since then $F=E\in\cyc(\M)$.
\begin{example}[Example~\ref{ex:A3syz} continued]
  The $\typeA_3$ arrangement has four proper cyclic flats given by the four
  triangles in the complete graph $K_4$.  They each have rank $2$, and so do
  their complements, so $\codim\frp_{F,F^\complement}=2+2$ for proper cyclic
  flats $F$, and $\codim\frp_{F,F^\complement}=3+0$ for $F=\emptyset$ and
  $F=E$.  So $\pdim S/\fra\geq4$, and our previous computation shows the bound
  is sharp in this case. It is also interesting to note that according to a computation using Macaulay2 \cite{M2}, $S/\fra$ has no embedded primes.
\end{example}

\begin{example}[\Cref{ex:7-hyps} continued]
  The cyclic flats are $\emptyset$,  $1246$, $1357$, and $[7]$, which have ranks $0$, $2$, $2$, and $3$, resp. Therefore,
  \[ \max\set{2\rank(F)-|F|+n-r\colon F\in \cyc(\M)} = \max\set{4,4,4,3} =4 \]
  However, a computation using Macaulay2 \cite{M2} shows that $\pdim_S(S/\fra)=7$.
\end{example}

For a subset $F\subseteq [n]$, let
\[P_F\coloneqq (f_i\colon i\in F),\]
an ideal of $R$.

\begin{corollary}\label{cor:min_primes2}
  Let $W$ be a matroid realization with no loops or coloops. 
    \begin{enumerate}[label=\textnormal{(\alph*)}]
    \item\label{cor:min_primes2_1} The associated primes of $(S/\fra)_{(\cdot,1)}$ are contained in 
      $\set{P_F \colon \emptyset \neq F\in L(\M)}$.
    \item\label{cor:min_primes2_2}
      The minimal primes of $(S/\fra)_{(\cdot,1)}$ are $\set{P_F\colon
      F\in \min(\cyc(\M)\setminus \{\emptyset\})}$.
    \end{enumerate}
\end{corollary}
\begin{proof}
    This follows immediately from \Cref{thm:min_primes}\ref{thm:min_1}, \ref{thm:min_2}, and \ref{thm:min_3} together with \Cref{ass-restr}.
\end{proof}

\begin{example}
  If $\M$ is a uniform matroid, then its only cyclic flats are $\emptyset$ and $[n]$. Therefore, by \Cref{cor:min_primes2}\ref{cor:min_primes2_2}, the only associated prime of $(S/\fra)_{(\cdot,1)}$ is $P_{[n]}$. This can also be proven using \Cref{lem:unif} below.
\end{example}

\begin{example}[\Cref{ex:bracelet-fra} continued]
  A computation using Macaulay2 \cite{M2} shows that neither $(S/\fra)_{(\cdot,1)}$ nor $(S/\fra)_{(1,\cdot)}$ have any embedded primes.
\end{example}

\begin{example}[\Cref{ex:7-hyps} continued]
  A computation using Macaulay2 \cite{M2} shows that, although $S/\fra$ has embedded primes, $(S/\fra)_{(\cdot,1)}$ does not. On the other hand,  $(S/\fra)_{(1,\cdot)}$ has does have an embedded prime, and this embedded prime is at the origin.
\end{example}

\begin{question}\label{ques:emb?}\hphantom{NEWLINE}
    \begin{enumerate}[label=\textnormal{(\alph*)}]
    \item What are the embedded primes of $(S/\fra)_{(\cdot,1)}$?  
    \item A priori, the flats $F$ for which $P_F$ is an embedded prime of $(S/\fra)_{(\cdot,1)}$ depend on the choice of realization. Are they, in fact, determined by the matroid $\M$?
    \end{enumerate}
\end{question}

\subsection{The uniform matroid}\label{subsec:uniform}
In \Cref{cor:unif} below, we prove the claim from \Cref{ex:unif} about the associated primes of $\fra$ when $\M$ is uniform. In order to do so, we will need the following \namecref{lem:unif}:

\begin{lemma}\label{lem:unif}
    If $\M=U_{r,n}$, then 
    \[g_{i_1}\cdots g_{i_r}f_a\in \fra\]
    for all $i_1,\ldots,i_r,a\in [n]$. In other words, $\frp_{\emptyset, [n]}^r W^*\subseteq \fra$.
\end{lemma}
\begin{proof}
    We proceed by induction on the number
    \[ \nu=\nu(i_1,\ldots,i_r) \coloneqq  \sum_{j\in [n]} \nu_j(i_1,\ldots,i_r),\]
    where
    \[\nu_j=\nu_j(i_1,\ldots,i_r) \coloneqq \begin{cases}
        \lvert\{k : i_k = j\}\rvert - 1,   &\text{if }j\in \{i_1,\ldots,i_r\},\\
        0,  &\text{otherwise.}
    \end{cases}\]
    
    If $\nu=0$, then all the $i_j$'s are distinct, and therefore by uniformity $f_a$ is in the $\k$-span of $f_{i_1},\ldots,f_{i_r}$, i.e.\ $f_a=c_1f_{i_1}+\cdots + c_rf_{i_r}$ for some $c_k\in \k$. So,
    $g_{i_1}\cdots g_{i_r}f_a\in \fra$.
    
    Suppose $\nu>0$. After re-indexing, we may assume that $i_1$ appears at least twice, i.e.\ $\nu_{i_1}>0$. Let $j_1,\ldots,j_{n-r}$ be distinct elements of $[n]\setminus \{i_1,\ldots,i_r\}$. By uniformity, $\{j_1,\ldots,j_{n-r}\}$ is a basis of $\M^\perp$, and therefore $g_{i_1}$ is in the $\k$-span of $g_{j_1},\ldots,g_{j_{n-r}}$. So,
    \[g_{i_1}\cdots g_{i_r}f_a \in \sum_{k=1}^{n-r} \k g_{j_k}g_{i_2}\cdots g_{i_r}f_a.\]
    By construction, $j_k$ appears exactly once in $j_k,i_2,\ldots,i_r$, and $i_1$ appears one fewer time than in $i_1,\ldots,i_r$. So, $\nu_{j_k}(j_k,i_2,\ldots,i_r)=0$, $\nu_{i_1}(j_k,i_2,\ldots,i_r)=\nu_{i_1}(i_1,\ldots,i_r)-1$, and no other $\nu_\ell$ is affected. Thus, 
    \[ \nu(j_k,i_2,\ldots,i_r) =\nu(i_1,i_2,\ldots,i_r)-1.\]
    Hence, by the induction hypothesis, $g_{j_k}g_{i_2}\cdots g_{i_r}f_a\in \fra$ for each $k$. Hence, $g_{i_1}\cdots g_{i_r}f_a\in \fra$.
\end{proof}

\begin{remark}
    In the proof of \Cref{lem:unif}, we showed in particular that if $i_1,\ldots,i_r\in [n]$ are distinct, then $g_{i_1}\cdots g_{i_r}f_a\in\fra$ for every $a\in [n]$. The same argument proves a more general statement: If $W$ is a realization of a matroid $\M$ and $B$ is a basis of $\M$, then for all $a\in [n]$, 
    \[ \left(\prod_{i\in B} g_i\right)f_a \in \fra.\]
    However, whether we can get something generalizing the full statement of \Cref{lem:unif} seems to be more subtle.
\end{remark}

\begin{corollary}\label{cor:unif}
    If $\M=U_{r,n}$, then $\frp_{[n],\emptyset}$, $\frp_{\emptyset,[n]}$ and $\frp_{[n],[n]}$ are the only associated primes of $S/\fra$.
\end{corollary}
\begin{proof}
    The only cyclic flats of $\M$ are $[n]$ and $\emptyset$, so by \Cref{thm:min_primes}\ref{thm:min_2}, the only minimal primes of $S/\fra$ are $\frp_{[n],\emptyset}$ and $\frp_{\emptyset,[n]}$. It therefore remains to show that $\frm=\frp_{[n],[n]}$ is the only other associated prime. 
    
    Let $\frp=\frp_{F,G}$ be an associated prime of $S/\fra$, and let $h\in S/\fra$ be homogeneous with $\operatorname{ann}_S(h)=\frp$. Let $u=\deg(h)$. If $u\in (1,1)+\NN^2$, then by \Cref{lem:unif}, $\frp_{[n],\emptyset}^{n-r} h=0$ and $\frp_{\emptyset, [n]}^{r} h=0$, which means that $\frp_{[n],\emptyset}$ and $\frp_{\emptyset,[n]}$ are both contained in $\frp$. Hence, $\frm$ is contained in and therefore equal to $\frp$. 
    
    Next, suppose $u\in \NN\times\{0\}$. Since $\fra_{(\cdot,0)}=0$, $h$ is not killed by any $f_i$. Therefore, $F=\emptyset$. Hence, by \Cref{thm:min_primes}\ref{thm:min_2}, $\frp=\frp_{\emptyset, [n]}$. Similarly, if $u\in \{0\}\times \NN$, then  $\frp=\frp_{[n], \emptyset}$.
\end{proof}

\subsection{Arrangements with non-isomorphic modules of derivations}
\label{subsec:noniso}
At first blush, there appears to be no particular reason to expect---or even hope!---that two different arrangements in the same ambient space will have isomorphic modules of derivations, and indeed they don't in general, as was shown by Ziegler in \cite[Example~8.7]{Zie89}. But then one encounters the class of free arrangements, all of which have (by definition) isomorphic modules of derivations. Maybe, one hopes, this phenomenon generalizes, at least within isomorphism classes of arrangements! Alas, as with so many properties of free arrangements, as soon as we move far enough away from freeness, we lose: In \Cref{recipe}, we give a recipe for producing combinatorially equivalent arrangements with non-isomorphic modules of derivations. Note that this recipe does not produce Ziegler's example \cite[Example~8.7]{Zie89}---in fact, as we explain in \Cref{Ziegler-fail}, this recipe is only useful when the rank is at least 4.

\begin{theorem}\label{recipe}
    Let $f,f'\colon W\hookrightarrow \k^n$ be realizations of a loopless matroid $\M$, and let $\calA$ and $\calA'$ be the induced arrangements in $W$. If there exists a minimal non-empty cyclic flat $F\in \cyc(\M)$ of rank at least 3 such that
    \[ \{w\in W : f_i(w)=0 \text{ for all } i\in F\} \neq \{w\in W : f'_i(w)=0 \text{ for all } i\in F\} ,\]
    then $\der(\calA)$ is not isomorphic to $\der(\calA')$.
\end{theorem}
\begin{proof}
    By \Cref{cor:min_primes2}\ref{cor:min_primes2_2}, the ideal 
    \[ P=(f_i : i\in F)\]
    is an associated prime of $(S/\fra)_{(\cdot,1)}$, and by assumption it is not an associated prime of $(S'/\fra')_{(\cdot,1)}$. Therefore, applying \cite[Th.~1.1(1)]{EHV92}, we find that $P$ is an associated prime of $\Ext^{\rank_\M(F)}_R((S/\fra)_{(\cdot,1)},R)$ but not of $\Ext^{\rank_\M(F)}_R((S'/\fra')_{(\cdot,1)},R)$. By \eqref{eq:3-terms-old}, $\der(\calA)$ is a second syzygy of $(S/\fra)_{(\cdot,1)}$. Therefore, since $\rank_\M(F)\geq 3$ by assumption,
    \[ \Ext^{\rank_\M(F)-2}_R(\der(\calA), R) \cong \Ext^{\rank_\M(F)}_R((S/\fra)_{(\cdot,1)},R).\]
    The same statement is true for $\calA'$. Thus, $P$ is an associated prime of $\Ext^{\rank_\M(F)-2}_R(\der(\calA), R)$ but not of $\Ext^{\rank_\M(F)-2}_R(\der(\calA'), R)$. Hence, $\der(\calA)$ is not isomorphic to $\der(\calA')$.
\end{proof}

\begin{remark}\label{Ziegler-fail}
    If the rank of $\M$ is less than or equal to 3, then the only cyclic flat of rank at least $3$ is $[n]$. But the flat of the arrangement cut out by $[n]$ is the origin irrespective of what realization of $\M$ we choose.
\end{remark}

\begin{example}\label{ex:fail}
    Consider the arrangement $\A$ in $\k^4$ given by the columns of the matrix
    \[A = \begin{pmatrix}
        0  &1  &0  &0  &1\\
        0  &0  &1  &0  &1\\
        0  &0  &0  &1  &0\\
        1  &1  &1  &1  &1
    \end{pmatrix}.\]
    The matroid of this arrangement is $\M = U_{1,\{4\}}\oplus U_{3,\{1,2,3,5\}}$. The only non-empty cyclic flat of $\M$ is $1235$, which has rank 3. Choose a $P\in \operatorname{GL}(\k^4)$ which does not fix $V(f_1,f_2,f_3,f_5)$. Then the columns of the matrix $PA$ give an arrangement $\A'$ combinatorially equivalent to $\A$, and the flats of $\A$ and $\A'$ corresponding to the flat $1235$ of $\M$ are distinct. Therefore, by \Cref{recipe}, $\der(\A)$ is not isomorphic to $\der(\A')$.
\end{example}

\begin{remark}
    \Cref{ex:fail} generalizes: Suppose $\M$ has a non-empty cyclic flat $F\neq [n]$ of rank at least 3, and suppose $\A$ is given by the columns of a matrix $A$. If there exists a $P\in \operatorname{GL}(W)$ which does not fix the flat of $\A$ corresponding to $F$ (e.g.\ if $\k$ is infinite), then the arrangement $\A'$ given by the columns of $PA$ has matroid $\M$, and by \Cref{recipe}, $\der(\A)\ncong \der(\A')$.
\end{remark}

\subsection{Applications to free arrangements}\label{subsec:free}
Various conditions on hyperplane arrangements have long been known which
are necessary for $\der(\A)$ to be free.  In this section, we note that
the support of the ideal of pairs (Section~\ref{ss:minprimes}) gives a
geometric explanation of a result of Kung and Schenck~\cite{KS06} which,
in turn, generalized a result of Ziegler~\cite{Zie89}.

\begin{theorem}[Cor.\ 2.3, \cite{KS06}]\label{th:pdim-leqc}
    Let $\A$ be a non-empty arrangement. Then
    \[ \pdim_R\der(\A) \geq \max \bigl\{\rank_\M F - 2 : F\in \min\bigl(\cyc(\M)\setminus \{\emptyset\}\bigr)\bigr\}.\]
\end{theorem}
\begin{proof}
    Assume that $\M$ has a minimal non-empty cyclic flat $F$ of rank $c\geq 3$. We are going to show that $\pdim_R\der(\A)\geq c-2$.
    
    Let $M= (S/\fra)_{(\cdot,1)}$. According to \eqref{eq:4-terms-old}, $\der(\A)(-1)$ is a second syzygy of $M$. Therefore, because $c\geq 3$,
    \[\Ext_R^{c-2}(\der(A)(-1),R) \cong \Ext_R^{c}(M,R).\]
    It therefore suffices to prove that $\Ext_R^{c}(M, R)\neq 0$. 
    
    By \Cref{cor:min_primes2}\ref{cor:min_primes2_2}, $P_F$ is an associated (in fact minimal) prime of $M$ of codimension $c$. So, \cite[Th.~1.1(1)]{EHV92} guarantees that $P_F$ is also an associated prime of $\Ext_R^{c}(M, R)$. Thus, the module $\Ext_R^{c}(M, R)\neq 0$, as claimed.
\end{proof}

This result is stated slightly differently in the original.  To translate, 
we note that $F$ is a minimal cyclic flat of $\M$ of rank $c$ if and only if
$\M|F$ is a uniform matroid.  Provided $c\geq2$, this is to say that 
$\A_F$ is a generic arrangement.

\begin{remark}
    Our proof of the the bound in \Cref{th:pdim-leqc} used only our knowledge of the minimal primes of $(S/\fra)_{(\cdot,1)}$. However, by \Cref{cor:min_primes}\ref{cor:min_primes2_1}, the same argument applied to all the associated primes shows that 
    \[ \pdim_R\der(\A) \geq \max\{\rank_\M F - 2 : P_F\text{ is an associated prime of }(S/\fra)_{(\cdot,1)}\}.\]
    This gives another motivation for \Cref{ques:emb?}.
\end{remark}

\begin{corollary}\label{th:free-leq2}
    Let $\A$ be an arrangement. If $\A$ is free, then every minimal non-empty cyclic flat has rank at most 2.
\end{corollary}

If $\A$ is a free simple arrangement, then, its minimal, non-empty cyclic flats
have rank exactly $2$.  Since the Boolean arrangement is the only arrangement
with no cyclic flats, this implies an early result of Ziegler:

\begin{corollary}[Cor.\ 7.6, \cite{Zie89}]\label{th:rk2}
  Let $\A$ be a simple arrangement.  If $\A$ is free and
  all rank-$2$ flats have size $2$, then $\A$ is the Boolean arrangement.
\end{corollary}

    
    

\begin{remark}
    The converse of \Cref{th:free-leq2} is false: All the minimal non-empty cyclic flats of the arrangement in \Cref{ex:bracelet} have rank 2, yet the arrangement is not free (in fact it is not even tame).
\end{remark}

\subsection{The case that \texorpdfstring{$\fra$}{a} is linear type}\label{sec:fra-lt}
Recall that an ideal $I$ in a Noetherian ring is \emph{linear type} if the natural map from the symmetric algebra $\Sym(I)$ of $I$ to the Rees algebra $\mathscr{R}(I)$ of $I$ is an isomorphism (see, e.g., \cite[\S2]{Hun80}). Unravelling the definitions involved, one gets \Cref{fra-lt} below.

To state the \namecref{fra-lt}, we will need the following two facts (both of which can be found in \cite[\S2]{Hun80}): The kernel of the natural map $S[a]\to \scrR(\fra)$ is
\begin{align*}
    \scrJ = \{h(a_1,\ldots,a_n) : h(f_1g_1,\ldots,f_ng_n)=0\}.
\end{align*}
The kernel of the natural map $S[a]\to \Sym(\fra)$ is the $S[a]$-ideal $\scrL$ generated by the elements in $\scrJ$ of $A$-degree $1$.

\begin{proposition}\label{fra-lt}\hphantom{newline}
    \begin{enumerate}[label=\textnormal{(\alph*)}]
        \item\label{fra-lt.sliceX} $I_\frX=\scrJ_{(\cdot, 1;\cdot)}$, i.e.\ $I_\frX$ is the part of $\scrJ$ with $R^\perp$-degree 1.
        \item\label{fra-lt.sliceLog} $I_{\log}=\scrL_{(\cdot,1;\cdot)}$, i.e.\ $I_{\log}$ is the part of $\scrL$ with $R^\perp$-degree 1.
        \item\label{fra-lt.main} If $\fra$ is linear type, then $I_{\log}=I_\X$.
    \end{enumerate}
\end{proposition}
\begin{proof}
    \ref{fra-lt.sliceX} This follows immediately from the definition of $I_\X$.
    
    \ref{fra-lt.sliceLog} By \Cref{prop:a-deg1}, the $A$-degree 1 part of $I_{\log}$ agrees with that of $I_\X$. Therefore, 
    \[ \scrL_{(\cdot,1;1)} = \scrJ_{(\cdot,1;1)} = (I_\X)_{(\cdot;1)} = (I_{\log})_{(\cdot;1)}.\]
    Now use that, by definition, $I_{\log}$ is generated in $A$-degree $1$.
    
    \ref{fra-lt.main} Since $\fra$ is linear type, $\scrJ=\scrL$. Now use parts \ref{fra-lt.sliceX} and \ref{fra-lt.sliceLog}.
\end{proof}

For tame arrangements, $I_{\log}=I_\X$ (\cite[Cor.~3.8]{CDFV11}). By \Cref{fra-lt}\ref{fra-lt.main}, this is also true if $\fra$ is linear type.  So we
have implications
\[
\begin{tikzcd}[column sep=small,row sep=small]
  \text{tame}\ar[rd,Rightarrow] & \\
  & I_{\log}=I_\X\\
  \text{$\fra$ is linear type}\ar[ru,Rightarrow] &
\end{tikzcd}
\]

It is therefore natural to ask the following questions:

\begin{question}\hphantom{newline}
    \begin{enumerate}[label=(\textnormal{\alph*})]
        \item Is the converse of \Cref{fra-lt}\ref{fra-lt.main} true?
        \item Is the converse of \cite[Cor.~3.8]{CDFV11} true?
        \item Does tameness imply that $\fra$ is linear type? What about the converse?
    \end{enumerate}
\end{question}

\bibliographystyle{amsalpha}
\bibliography{arrangements,pairs}
\end{document}